\documentclass[12pt,notitlepage]{amsart}
\usepackage{txfonts}
\usepackage{subfig}
\usepackage{amscd,amssymb,mathabx}
\usepackage[arrow,matrix,graph,frame,poly,arc,tips]{xy}
\usepackage{graphicx,calrsfs}
\usepackage{mathtools}
\usepackage{tikz}
\usetikzlibrary{decorations.markings}
\usetikzlibrary{shapes}
\usetikzlibrary{backgrounds}
\usetikzlibrary{calc}
\usepackage{mdwlist}
\usepackage{multirow}
\usepackage{enumerate}
\usepackage{verbatim}
\usepackage[ngerman,french,english]{babel}
\usepackage[utf8]{inputenc}

\DeclarePairedDelimiter{\form}{\langle}{\rangle}

\newcommand\ba{\begin{align*}}
\newcommand\ea{\end{align*}}
\newcommand\be{\begin{enumerate}}
\newcommand\ee{\end{enumerate}}
\newcommand\bp{\begin{proof}}
\newcommand\ep{\end{proof}}
\newcommand\bpp{\begin{prop}}
\newcommand\epp{\end{prop}}
\newcommand\bpb{\begin{prob}}
\newcommand\epb{\end{prob}}
\newcommand\bd{\begin{defn}}
\newcommand\ed{\end{defn}}
\newcommand\bh{\begin{hint}}
\newcommand\eh{\end{hint}}

\newcommand\fform[1]{\langle\!\langle #1\rangle\!\rangle}

\newcommand\bR{\mathbb{R}}
\newcommand\R{\mathbb{R}}

\newcommand\bZ{\mathbb{Z}}
\newcommand\Z{\mathbb{Z}}

\renewcommand\AA{\mathcal{A}}

\newcommand\CC{\mathcal{C}}
\newcommand\FF{\mathcal{F}}
\newcommand\GG{\mathcal{G}}

\newcommand\NN{\mathcal{N}}
\newcommand\OO{\mathcal{O}}

\newcommand\supp{\operatorname{supp}}
\newcommand\Id{\operatorname{Id}}

\DeclareMathOperator\Homeo{Homeo}

\newcommand\sse{\subseteq}
\newcommand\co{\colon}

\DeclareMathOperator\Diff{Diff}

\renewcommand{\MR}[1]
{\href{http://www.ams.org/mathscinet-getitem?mr=#1}{MR#1}}

\def\thetitle{{Chain groups of homeomorphisms of the interval}}
\def\theauthors{{Sang-hyun Kim and Thomas Koberda and Yash Lodha}}
\usepackage{hyperref}
\hypersetup{
   colorlinks=false,
   plainpages,
   urlcolor=black,
   linkcolor=black
   pdftitle=   \thetitle,
   pdfauthor=  {\theauthors}
}

\theoremstyle{theorem}
\newtheorem{thm}{Theorem}[section]
\newtheorem{lem}[thm]{Lemma}
\newtheorem{cor}[thm]{Corollary}
\newtheorem{prop}[thm]{Proposition}

\newtheorem*{claim*}{Claim}

\newtheorem{claim}{Claim}

\newtheorem{setting}{Setting}[section]

\theoremstyle{remark}

\newtheorem{rem}[thm]{Remark}

\theoremstyle{definition}
\newtheorem{defn}[thm]{Definition}
\newtheorem{prob}{Problem}[section]

\begin{document}
\title\thetitle

\begin{abstract} 
   We introduce and study the notion of a chain group of homeomorphisms of a one-manifold, which is a certain generalization of Thompson's group $F$. The resulting class of groups exhibits a combination of uniformity and diversity. On the one hand, a chain group either has a simple commutator subgroup or the action of the group has a wandering interval. In the latter case, the chain group admits a canonical quotient which is also a chain group, and which has a simple commutator subgroup. On the other hand, every finitely generated subgroup of $\operatorname{Homeo}^+(I)$ can be realized as a subgroup of a chain group. As a corollary, we show that there are uncountably many isomorphism types of chain groups, as well as uncountably many isomorphism types of countable simple subgroups of $\operatorname{Homeo}^+(I)$. We consider the restrictions on chain groups imposed by actions of various regularities, and show that there are uncountably many isomorphism types of $3$--chain groups which cannot be realized by $C^2$ diffeomorphisms, as well as uncountably many isomorphism types of $6$--chain groups which cannot be realized by $C^1$ diffeomorphisms. As a corollary, we obtain uncountably many isomorphism types of simple subgroups of $\operatorname{Homeo}^+(I)$ which admit no nontrivial $C^1$ actions on the interval. Finally, we show that if a chain group acts minimally on the interval, then it does so uniquely up to topological conjugacy.
\end{abstract}


\author[S. Kim]{Sang-hyun Kim}
\address{Department of Mathematical Sciences, Seoul National University, Seoul, 08826, Korea}
\email{s.kim@snu.ac.kr}
\urladdr{http://cayley.kr}

\author[T. Koberda]{Thomas Koberda}
\address{Department of Mathematics, University of Virginia, P. O. Box 400137, Charlottesville, VA 22904-4137, USA}
\email{thomas.koberda@gmail.com}
\urladdr{http://faculty.virginia.edu/Koberda}

\author[Y. Lodha]{Yash Lodha}
\address{EPFL SB MATH EGG, Station 8, MA C3 584 (B\^atiment MA), Station 8,  CH-1015, Lausanne, Switzerland}
\email{yash.lodha@epfl.ch}
\urladdr{https://people.epfl.ch/yash.lodha}

\maketitle


\section{Introduction}\label{sec:intro}
In this paper, we introduce and study the notion of a chain group of homeomorphism of a connected one--manifold. A chain group can be viewed as a generalization of Thompson's group $F$ which sits inside of the group of homeomorphisms of the manifold in a particularly nice way. 

We denote by $\Homeo^+(\bR)$ the group of orientation preserving homeomorphisms on $\bR$.
The \emph{support} of $f\in \Homeo^+(\bR)$ is the set of $x\in \bR$ such that $f(x)\ne x$.
The support of a group $G\le\Homeo^+(\bR)$ is defined as the union of the supports of all the elements in $G$.
For an interval $J\sse \bR$, let us denote the left-- and the right--endpoints of $J$
 by $\partial^-J$ and $\partial^+ J$, respectively.

Suppose $\mathcal{J}=\{J_1,\ldots,J_n\}$ is a collection of nonempty open subintervals of $\bR$. 
We call $\mathcal{J}$ a \emph{chain of intervals} (or an \emph{$n$--chain of intervals} if the cardinality of $\mathcal{J}$ is important) if $J_i\cap J_k=\varnothing$ if $|i-k|>1$, and if $J_i\cap J_{i+1}$ is a proper nonempty subinterval of $J_i$ and $J_{i+1}$ for $1\leq i\leq n-1$. See Figure~\ref{f:coint}.
\begin{figure}[h!]
\begin{tikzpicture}[ultra thick,scale=.5]
\draw [red] (-5,0) -- (-1,0);
\draw (-3,0) node [above] {\small $J_1$}; 
\draw [blue] (-2,.5) -- (2,.5);
\draw (0,.5) node [above] {\small $J_2$}; 
\draw [teal] (1,0) -- (5,0);
\draw (3,0) node [above] {\small $J_3$}; 
\end{tikzpicture}
\caption{A chain of three intervals.}
\label{f:coint}
\end{figure}
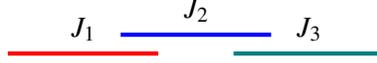

\begin{setting}\label{set}
We let $n\ge2$ and let $\mathcal{J}=\{J_1,\ldots,J_n \}$ 
be a chain of intervals such that $\partial^- J_i<\partial^- J_{i+1}$ for each $i<n$.
We consider a collection of homeomorphisms $\mathcal{F}=\{f_1,\ldots,f_n \}$ such that  $\supp f_i=J_i$
and such that $f_i(t)\ge t$ for each $t\in\bR$.
We set $G_{\mathcal{F}}=\form{\FF}\le\Homeo^+(\bR)$.
\end{setting}

We call the group $G=G_\FF$ a \emph{prechain group}. We say that $G$ is a \emph{chain group} (sometimes \emph{$n$--chain group}) if, moreover,
the group $\form{f_i,f_{i+1}}$ is isomorphic to Thompson's group $F$
 for each $i=1,2,\ldots,n-1$.
Whereas this definition may seem rather unmotivated, part of the purpose of this article is to convince the reader that chain groups are natural objects.

In this paper, we will consider chain groups as both abstract groups and as groups with a distinguished finite generating set $\mathcal{F}$ as above. Whenever we mention ``the generators" of a chain group, we always mean the distinguished generating set $\mathcal{F}$ which realizes the group as a chain group of homeomorphisms.

\subsection{Main results}
Elements of the class of chain groups enjoy many properties which are mostly independent of the choices of the homeomorphisms generating them, and at the same time can be very diverse. Moreover, chain groups are abundant in one--dimensional dynamics. Our first result establishes the naturality of chain groups:

\begin{thm}\label{thm:dynamical condition}
If $G$ is a prechain  group as in Setting~\ref{set}, then
the group \[G_N:=\form{ f^N\mid f\in\mathcal{F}}\le \Homeo^+(\bR)\] is a chain group
for all sufficiently large $N$.
\end{thm}

In Theorem~\ref{thm:dynamical condition} and throughout this paper, $N$ a sufficiently large exponent means that $N$ is larger than some natural number which depends on the particular generators of the chain group.

Choosing a two--chain group whose generators are $C^{\infty}$ diffeomorphisms of $\bR$, we obtain the following immediate corollary, which is a complement to (and an at least partial recovery of) a result of Ghys--Sergiescu~\cite{GhysSergiescu}:

\begin{cor}\label{cor:smooth}
Thompson's group $F$
 can be realized as a subgroup of $\Diff^{\infty}(I)$, the group of $C^{\infty}$ orientation preserving diffeomorphisms of the interval.
\end{cor}

Let $M$ be a one--manifold.
We say $G\le\Homeo^+(M)$ is \emph{minimal} if every orbit is dense in $\supp G=\cup_{g\in G}\supp g$.
General chain groups have remarkably uncomplicated normal subgroup structure, in relation to their dynamical features:

\begin{thm}\label{thm:simple}
For an $n$--chain group $G$, exactly one of the following holds:
\begin{enumerate}[(i)]
\item
The action of $G$ is minimal;
in this case, 
every proper quotient of $G$ is abelian
and
the commutator subgroup $G'\le G$ is simple;
\item
The closure of some $G$--orbit is perfect and totally disconnected;
in this case, 
$G$ canonically surjects onto an $n$--chain group that acts minimally.
\end{enumerate}
\end{thm}

A general chain group may fail to have a simple commutator subgroup (Proposition~\ref{prop:blow-up}).

Every finitely generated subgroup of $\Homeo^+(\bR)$ embeds into a minimal chain group, so that the subgroup structure of chain groups can be extremely complicated:

\begin{thm}\label{thm:embed}
Let $G=\form{f_1,f_2,\ldots,f_n}\le \Homeo^+(\bR)$ for some $n\ge2$.
\be
\item
Then $G$ embeds into an $(n+2)$--chain group.
\item
If $\supp f_1$ has finitely many components,
then $G$ embeds into an $(n+1)$--chain group.
\ee
\end{thm}

The notion of \emph{rank} of a chain group is somewhat subtle. Indeed, the next proposition shows that a given $n$--chain group not only contains $m$--chain groups for all $m$, but is in fact isomorphic to an $m$--chain group for all $m\geq n$:

\begin{prop}\label{prop:subchain group}
For $m\ge n\geq 2$, each $n$--chain group is isomorphic to some $m$--chain group.
\end{prop}

Thus, we are led to define the rank of an $n$--chain group $G$ to be the minimal $n$ for which $G$ is isomorphic to an $n$--chain group. In light of Proposition~\ref{prop:subchain group}, it is not clear if there exist chain groups of all ranks, though there in fact are:

\begin{prop}\label{prop:abelianization}
For all integers $2\leq k\leq n$, there exists an $n$--chain group $G_k$ such that $H_1(G_k,\bZ)\cong\bZ^k$.
\end{prop}

Clearly an $n$--chain group whose abelianization is isomorphic to $\bZ^n$ cannot have rank less than $n$.

Throughout this paper, we will say that a group is \emph{infinitely presented} if it is not finitely presentable. Theorem \ref{thm:embed} has several consequences which we note here.
By definition, there is only one isomorphism type of $2$--chain groups.
This is in stark contrast to the case $n\ge3$, where
there exist uncountably many isomorphism types of $n$--chain groups. More precisely, we have the following:

\begin{thm}\label{thm:uncountable}
Let $n\ge3$. Then there exists an uncountable set $\GG$ of infinitely presented $n$--chain groups with simple commutator subgroups such that $G'\not\cong H'$ for all distinct $G,H\in\GG$.\end{thm}
 
In the case $n=3$, Thereom \ref{thm:uncountable} answers a question posed to the authors by J. McCammond. We deduce the following:

\begin{cor}\label{cor:simple}
For every one--manifold $M$, there exist uncountably many isomorphism types of countable simple subgroups of $\Homeo^+(M)$. These simple groups can be realized as the commutator subgroups of $3$--chain groups.
\end{cor}

Specializing to the interval, we will deduce the following:

\begin{cor}\label{cor:c1}
There exist uncountably many isomorphism types of countable simple subgroups of $\Homeo^+(I)$ which admit no nontrivial $C^1$ action on $I$.
\end{cor}

The simple subgroups in Corollary \ref{cor:simple} will necessarily be infinitely generated.
We remark that it is not difficult to establish Corollary \ref{cor:simple} for general $n$--manifolds, though we will not require such a statement here.

Theorem \ref{thm:uncountable} and its corollaries above show that chain  groups can be very diverse. However, chain groups exhibit a remarkable phenomenon called \emph{stabilization} 
(see Section \ref{sec:dyn} for precise definitions and discussion). 
We have the following results, which show that chain groups groups acquire certain \emph{stable isomorphism types}:

\begin{prop}\label{prop:stabilization-chain}
Let $G=G_{\mathcal{F}}$ be an $n$--chain group for $n\geq 2$. Then for all sufficiently large $N$, the group $G_N=\langle f^N\mid f\in\mathcal{F}\rangle$ is isomorphic to the Higman--Thompson group $F_n$.
\end{prop}

The Higman--Thompson groups $\{F_n\}_{n\geq 2}$ are defined and discussed in Subsection \ref{ss:thompson} below. We remark that Proposition \ref{prop:stabilization-chain} was observed independently (and in a different context) by Bleak--Brin--Kassabov--Moore--Zaremsky~\cite{BBKMZ16}.

Finally, we show that chain groups are more or less uniquely realized as groups of homeomorphisms of the interval. We will say that a chain group $G<\Homeo^+(I)$ acts \emph{minimally} if for all $x\in (0,1)$, we have that the orbit $G. x$ is dense in $I$.

\begin{thm}\label{thm:uniqueness}
Let $\alpha,\beta\colon G\to\Homeo^+(I)$ be two actions of a chain group $G$ which are minimal. Then there exists a homeomorphism $h\colon I\to I$ such that for all $g\in G$ and $x\in I$, we have \[\alpha(g). x=h\circ\beta(g)\circ h^{-1}(x).\]
\end{thm}

It is essential to note that Theorem~\ref{thm:uniqueness} applies only to chain group actions \emph{as chain groups}, and not as abstract groups. Two abstract minimal actions of a chain group are topologically conjugate provided they are both \emph{locally dense}. See Section~\ref{sec:uniqueness}. Theorem~\ref{thm:uniqueness} was already known in several contexts, such as for Thompson's group itself and more generally for the Higman--Thompson groups (cf.~\cite{GhysSergiescu,BBKMZ16}).

\subsection{Notes and references}

\subsubsection{Motivations}
The original motivation for studying chain groups comes from the first and second authors' joint work with Baik~\cite{BKK16} (cf.~\cite{KK2017}) which studied right-angled Artin subgroups of $\Diff^2(S^1)$. If $G=G_{\mathcal{F}}<\Homeo^+(I)$ is a prechain group, then one can form a graph whose vertices are elements of $\mathcal{F}$, and where an edge is drawn between two vertices if the supports are disjoint. The graph one obtains this way is an \emph{anti--path}, i.e. a graph whose complement is a path. In light of general stabilization results for subgroups generated by ``sufficiently high powers" of homeomorphisms or of group elements or of mapping classes as occur in right-angled Artin group theory and mapping class group theory (see for instance~\cite{CLM2012,Koberda2012,KK2013,KK2013b,KK2015}), one might guess that after replacing elements of $\mathcal{F}$ by sufficiently high powers, one would obtain the right-angled Artin group on the corresponding anti--path. It is the somewhat surprising stabilization of a $2$--chain group at the isomorphism type of Thompson's group $F$ instead of the free group on two generators that provides the key for systematically approaching the study of chain groups. Finally, Proposition \ref{prop:stabilization-chain}  gives the correct analogous general stabilization of isomorphism type result for general chain groups.

The algebraic structure of $2$--prechain groups, which by definition stabilize to become isomorphic to Thompson's group $F$, have been investigated by the second and third author in~\cite{KL2017}.

\subsubsection{Uncountable families of countable simple groups}

Within any natural class of countable (and especially finitely generated) groups, it is typical to encounter a countable infinity of isomorphism types, but uncountable infinities of isomorphism types often either cannot be exhibited, or finding them is somewhat nontrivial. Within the class of finitely generated groups, the Neumann groups provide some of the first examples of an uncountable family of distinct isomorphism types of two--generated groups. The reader may find a description of the Neumann groups in Section \ref{sec:uncountable}, where we use these groups allow us to find uncountably many isomorphism classes of chain groups.

Uncountable families of simple groups, or more generally groups with a specified property, are also difficult to exhibit in general. Uncountably many distinct isomorphism types of finitely generated simple groups can be produced as a consequence of variations on~\cite{BaumslagRoseblade}. For other related results on uncountable families of countable groups satisfying various prescribed properties, the reader is directed to~\cite{MonodShalom,OlshanskiiOsin,BCGS14,Mikaelian,Pyber,Brookes,HigmanScott,Leary2015}, for instance. To the authors' knowledge, the present work produces the first examples of uncountable familes of distinct isomorphism types of countable subgroups of $\Homeo^+(I)$, as well as an uncountable families of distinct isomorphism types of countable simple subgroups of $\Homeo^+(I)$ (and more generally of $\Homeo^+(M)$ for an arbitrary manifold $M$).

\section{Preliminaries}\label{s:prelim}
We gather some well--known facts here which will be useful to us in the sequel.
\emph{Throughout this paper,
we fix the notations
\[
a(x)=x+1\quad\text{and}\quad
b(x)=
\begin{cases}
 x&\text{ if }x\leq 0,\\
 2x&\text{ if }0<x< 1,\\
  x+1&\text{ if }1\leq x\\
\end{cases}
\]
}
and let $F = \form{a,b}\le\Homeo^+(\bR)$ be the \emph{standard} copy of Thompson's group $F$.

\subsection{Thompson's group and Higman--Thompson group}\label{ss:thompson}
The canonical references for this section are the classical Cannon--Floyd--Parry notes~\cite{CFP96} and the contemporary treatment in Burillo's in--progress book~\cite{BurilloBook}.  

The group $F$ was originally defined as the group of piecewise linear orientation preserving homeomorphisms of the unit interval with dyadic breakpoints, and where all slopes are powers of two.
With the maps $A,B\in\mathrm{PL}^+(I)$ illustrated in Figure~\ref{fig:gens}, we can write a presentation
 \[F=\form{A,B \mid [AB^{-1}, A^{-k}BA^k]\text{ for }k=1,2}.\]
 
Using $h\in\mathrm{PL}^+(\bR)$ in Figure~\ref{fig:gens} (c), the maps $a,b$ introduced above
 satisfy  \[a=hA^{-1}h^{-1},\quad b = hB^{-1}h^{-1},\]
and hence,
\[ F=\form{a,b\mid [b^{-1}a,a^kba^{-k}]\text{ for }k=1,2}
=\form{u,b|[u,(bu)^k b(bu)^{-k}]\text{ for }k=1,2}.\]

Thompson's group $F$ has a very diverse array of subgroups. An example of a finitely generated subgroup of $F$ which is not finitely presented is the lamplighter group 
$\Z\wr\Z$. We record this fact for use later in the paper:

\begin{lem}\label{lem:lamplighter}
Thompson's group $F$ contains a copy of the lamplighter group $\Z\wr\Z$.
\end{lem}
\bp
Pick a compactly supported nontrivial element $f\in F$. 
Then for all $N$ sufficiently large, it is easy to check that $\langle a^N,f\rangle\cong\Z\wr\Z$.\ep

\begin{figure}[b!]
  \tikzstyle {a}=[postaction=decorate,decoration={%
    markings,%
    mark=at position .65 with {\arrow{stealth};}}]
\subfloat[(a) $A$]{\begin{tikzpicture}[scale=.75]
\draw (0,0) -- (0,.5) -- (4,.5) -- (4,0) --cycle;
\draw [a] (2,.5) -- (1,0);
\draw [a] (3,.5) -- (2,0);
\draw [right] (0,.5) node [above] {\tiny $0$};
\draw [right] (4,.5) node [above] {\tiny $1$};
\draw  (2,0.5) node  [above] {\tiny $\frac12$} (3,.5) node [above] {\tiny $\frac34$};
\draw [below]  (1,0) node   {\tiny $\frac14$} (2,0) node {\tiny $\frac12$};
\draw [below]  (4,0) node   {\tiny $1$};
\draw [below]  (0,0) node   {\tiny $0$};
\end{tikzpicture}}
\quad\quad
\subfloat[(b) $B$]{\begin{tikzpicture}[scale=.75]
\draw (0,0) -- (0,.5) -- (4,.5) -- (4,0) --cycle;
\draw [a] (2,.5) -- (2,0);
\draw [a]  (3,.5)--(2.5,0) ;
\draw [a]  (3.5,.5)--(3,0) ;
\draw [right] (0,.5) node [above] {\tiny $0$};
\draw [right] (4,.5) node [above] {\tiny $1$};
\draw [below]  (0,0) node   {\tiny $0$};
\draw [below]  (4,0) node   {\tiny $1$};
\draw  [below] (2,0) node {\tiny $\frac12$} (3,0) node  {\tiny $\frac34$} (2.5,0) node {\tiny $\frac58$};
\draw [above]  (2,.5) node {\tiny $\frac12$} (3,.5) node   {\tiny $\frac34$} (3.5,.5) node   {\tiny $\frac78$};
\end{tikzpicture}}
\quad\quad
\subfloat[(c) $h$]{\begin{tikzpicture}[scale=.75]
\draw (0,0) -- (0,.5) -- (4,.5) -- (4,0) --cycle;
\foreach \i in {1,3,3.5} \draw [a] (\i,.5)--(\i,0);
\draw [right] (0,.5) node [above] {\tiny $0$};
\draw [right] (4,.5) node [above] {\tiny $1$};
\draw [below]  (0,0) node   {\tiny $-\infty$};
\draw [below]  (4,0) node   {\tiny $\infty$};
\draw  [below] (1,0) node {\tiny $-1$};
\draw  [below] (3,0) node  {\tiny $1$} (3.5,0) node  {\tiny $2$} ;
\draw [above]  
(1,.5) node {\tiny $\frac14$} (3,.5) node   {\tiny $\frac34$} (3.5,.5) node   {\tiny $\frac78$};
\end{tikzpicture}}
\caption{Elements $A,B,h$ of Thompson's group $F$.}
\label{fig:gens}
\end{figure}
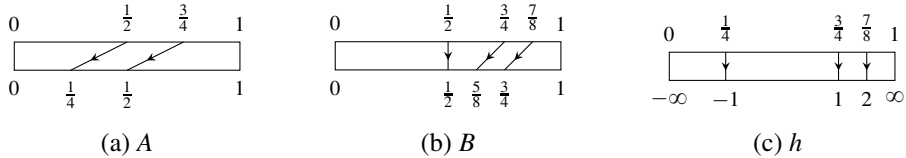

The Higman--Thompson groups $\{F_n\}_{n\geq 2}$ are certain generalizations of Thompson's group $F$, which first appear in~\cite{HigmanBook} (see also~\cite{Brown1985}, where the notation $F_{n,1}$ is used for $F_n$). 
Their relationship to Thompson's group $F$ is most evident from the following well--known presentation, which the reader may find in~\cite{BurilloBook} for instance: \[F\cong\langle \{g_i\}_{i\in\Z_{\geq 0}}\mid g_j^{g_i}=g_{j+1}\textrm{ if $0\le i<j$}\rangle.\]

The group $F_n$ can be defined by the following infinite presentation: \[F_n=\langle \{g_i\}_{i\in\Z_{\geq 0}}\mid g_j^{g_i}=g_{j+n-1}\textrm{ if $0\le i<j$}\rangle.\] Observe that $F_2$ is just Thompson's group $F$. The groups $\{F_n\}_{n\geq 2}$ have many remarkable algebraic features, the essential ones of which we record here:

\begin{lem}\label{lem:brown1}
For each $n\ge2$, the group $F_n$ has a trivial center and a simple commutator subgroup.
\end{lem}

It follows that $F_n''=F_n'$ and that every proper quotient of $F_n$ is abelian.

\subsection{Higman's Theorem}\label{subsec:higman}

The main technical tool we will require to prove Theorem \ref{thm:simple} is Higman's Theorem~\cite{Higman}. To properly state this result, consider a group $H$ acting on a set $X$. For $g\in H$, we will retain standard notation and write \[\supp g=\{x\in X\mid g(x)\neq x\}.\]

\begin{thm}[Higman]\label{thm:higman}
Let $H$ be a group acting faithfully on a set $X$, and suppose that for all triples $r,s,t\in H\setminus \{1\}$ there is an element $u\in H$ such that \[(\supp r\cup\supp s)\cap (t^u(\supp r\cup\supp s))=\varnothing.\] Then the commutator subgroup $H'=[H,H]$ is simple.
\end{thm}

The following definitions will be useful when  Higman's Theorem is applied.

\bd\label{defn:cpt-trans}
Let $X$ be a topological space, and let $H\le\Homeo(X)$.
\be
\item
We say $H$ acts \emph{CO-transitively} (or, \emph{compact--open-transitively}) if
for each proper compact subset $A\sse X$ and for each nonempty open subset $B\sse X$,
there is $u\in H$ such that $u(A)\sse B$.
\item
We say $H$ acts \emph{locally CO-transitively} if
for each proper compact subset $A\sse X$ there exists some point $x\in X$
such that for each open neighborhood $B$ of $x$
there is $u\in H$ such that $u(A)\sse B$.
\ee
\ed

We shall use the following variation of Theorem \ref{thm:higman}; see also ~\cite{Epstein1970,DS1991Glasgow,BieriStrebel}.

\begin{lem}\label{lem:higman}
Let $X$ be a non-compact Hausdorff space, and let $\Homeo_c(X)$ denote the group of compactly supported homeomorphisms.
If $H\le\Homeo_c(X)$ is CO-transitive,
then $H'$ is simple.
\end{lem}

\bp
In order to apply Higman's Theorem, we fix $r,s,t\in H\setminus\{1\}$.
Since $X$ is Hausdorff, we can choose a nonempty open set $B\sse X$ such that $B\cap tB=\varnothing$.
Choose a proper compact subset $A$ containing $\supp r\cup\supp s$.
By CO-transitivity, there is $u\in H$ such that $uA\sse B$.
From $uA\cap tuA=\varnothing$, we see
\[
u(\supp r\cup \supp s)\cap tu(\supp r\cup \supp s)=\varnothing,\]
as desired in the condition of Higman's Theorem.
\ep

Recall a group action on a topological space is \emph{minimal} if every orbit is dense. 

\begin{lem}\label{lem:cpt-trans}
A minimal, locally CO-transitive group action is CO-transitive.
\end{lem}

\bp
Let $G$ be the given group acting on a space $X$,
 let $A\sse X$ be proper and compact, and let $B\sse X$ be open.
Suppose $x$ is an accumulation point of the $G$--orbit of $A$, given by the hypothesis.
By minimality, we can find an open neighborhood $J$ of $x$ and an element $s\in G$ such that $s(J)\sse B$.  Then there is $r\in G$ such that $r(A)\sse J$
and $u=sr$ is a desired element.
\ep

\subsection{Left--orderability and subgroups of $\Homeo^+(\R)$}
A \emph{left--order} on a group $G$ is a total order $<$ on $G$ which is left invariant, i.e. for all $g,h,k\in G$ we have $h< k$ if and only if $gh < gk$.
A group is \emph{left--orderable}, if it can be equipped with some left--order.
While orderability is an algebraic property of a group, it has a very useful dynamical interpretation. The reader may find the following fact as Theorem 2.2.19 of \cite{Navas2011} (see also~\cite{DeroinNavasRivas}):

\begin{lem}\label{lem:dynamical realization}
Let $G$ be a countable group. We have that $G<\Homeo^+(\R)$ (equivalently $G<\Homeo^+(I)$) if and only if $G$ is left orderable.
\end{lem}

Let $A$ and $B$ be groups equipped with left--orders.
A homomorphism $f\co A\to B$ is \emph{monotone increasing} 
if for every $g<h$ in $A$ we have $f(g)\le f(h)$. We have the following well-known fact, whose proof we omit:

\begin{lem}\label{lem:extension}
Suppose we have a short exact sequence of groups
\[\xymatrix{
1\ar[r]&
N\ar[r]^i&
G\ar[r]^p&
Q\ar[r]&
1
}.\]
If $N$ and $Q$ are equipped with left--orders, then there uniquely exists a left--order on $G$ such that $i$ and $p$ are monotone increasing.
\end{lem}

\section{Dynamical aspects}\label{sec:dyn}
\subsection{Stabilization}\label{ss:stab}
Let us note a dynamical condition that guarantees a prechain group to be a chain group.

\begin{lem}[cf. Section~\ref{sec:uniqueness}]\label{lem:dynamical condition}
The following are true.
\be
\item
Suppose the supports of $f,g\in\Homeo^+(\bR)$ are given by
\[\supp f=(x,z),\quad\supp g=(y,w)\]
for some $-\infty\le x<y<z<w\le\infty$.
If $gf(y)\ge z$,
then 
$\form{f,g}\cong F$.
\item
If a prechain group $G$ as in Setting~\ref{set} satisfies
\[f_n\cdots f_1(\supp f_1\setminus\supp f_2)\cap (\supp f_n\setminus\supp f_{n-1})\ne\varnothing,\]
then $G$ is isomorphic to the Higman--Thompson group $F_n$.
\ee
\end{lem}

\bp
(1)
Note that $f$ and $g$ move points to the right.
For $k\ge1$ we have that
\[
\supp f\cap \supp (gf)^k g (gf)^{-k}
=\supp f\cap (gf)^k\supp  g =\varnothing.\]
So $F$ surjects onto $G=\form{f,g}$. Since every proper quotient of $F$ is abelian and $G$ is nonabelian, we see $F\cong G$.

(2)
For $0\le i\le n-1$ and for $q\ge0$, 
we define 
\[h_i=f_{i+1}^{-1} \cdots f_{n-1}^{-1} f_n^{-1},\quad 
h_{i+q(n-1)}=h_0^{-q} h_i h_0^q,\] 
so that $\langle h_0,\ldots,h_{n-1}\rangle= G$. 
By the same idea as in (1), we verify the relations of $F_n$ in $G$:
\[
h_j^{h_i}=h_{j+n-1}\text{ for }0\le i<j.\]
Since every proper quotient of $F_n$ is abelian, we have $G\cong F_n$.\ep

Let $G=G_{\mathcal{F}}$ be a prechain group. We say that $G$ \emph{stabilizes} if for all $N$ sufficiently large, the groups $\{G_N=\langle \{f^N\mid f\in\mathcal{F}\}\rangle\}$ form a single isomorphism class, called the \emph{stable type} of $G$. 
The following asserts that an $n$--prechain group stabilizes for all $n\ge2$.

\begin{thm}\label{thm:dynamical}
For a prechain group $G_\FF$ as in Setting~\ref{set},
the group $G_N:=\form{f^N\mid f\in\mathcal{F}}$ is a chain group isomorphic to $F_n$,
whenever $N$ is sufficiently large. \end{thm}

\begin{proof}
Let $G=G_{\mathcal{F}}$ be a prechain group. By replacing the elements of $\mathcal{F}$ by sufficiently high positive or negative powers, we may assume that the dynamical conditions of Lemma~\ref{lem:dynamical condition} hold.
The resulting group generated by these powers of homeomorphisms will therefore be the desired chain group.
\end{proof}

For groups $H\le G$, the normal closure of $H$ in $G$ is denoted as $\fform{H}_G$.

\begin{lem}\label{lem:center}
For every chain group $G$, we have $Z(G)=\{1\}$ and $G''=G'$.
\end{lem}

\bp
Let $G=G_\FF$ be a chain group for $\FF=\{f_1,\ldots,f_n\}$
as in Section~\ref{sec:intro}. We may assume $\supp G=\bR$.
If $1\ne g\in Z(G)$, then $\supp g$ is $G$--invariant and accumulates at $\pm\infty$. 
Since the germ of $G$ at $-\infty$ is $\form{f_1}$, we see the restriction of $g$ on $(-\infty,t)$ for some $t\in\bR$
coincides with $f_1^k$ for some $k\ne0$.
By the $G$--invariance, we see $\supp g=\bR$. 
As $\supp {f_1}$ is $\form{g}$--invariant, we have a contradiction.

The second part of the lemma follows from
\begin{align*}
&[f_i,f_{i+1}]\in \form{f_i,f_{i+1}}'=\form{f_i,f_{i+1}}''\le G''\unlhd G,\\ 
&G'=\fform{\{[f_i,f_{i+1}]\co 1\le i<n\}}_G\le G''.\qedhere
\end{align*}
\ep

\subsection{A dynamical dichotomy}
Let $G$ be a group faithfully acting on $\bR$.
A closed nonempty $G$--invariant set $\Lambda\sse\bR$ is called a \emph{minimal invariant set} of $G$ if no proper nonempty closed subset of $\Lambda$ is $G$--invariant.
Such a set $\Lambda$ exists in the case when $G$ is finitely generated~\cite{Navas2011}, 
but it may not be unique.

Assume a minimal invariant set $\Lambda$ of $G\le\Homeo^+(\bR)$ exists.
We say $G$ is \emph{minimal} if $\Lambda=\bR$.
If $\Lambda$ is perfect and totally disconnected,
then $\Lambda$ is called an \emph{exceptional} minimal invariant set.  

Let us first note two lemmas on general dichotomy of one--dimensional homeomorphism groups;
the proofs are variations of~\cite[Section 2.1.2]{Navas2011}, which we omit.

\begin{lem}\label{lem:dichotomy}
Let $G\le\Homeo^+(\bR)$ be a group such that $\supp G=\bR$.
\be
\item
If all the $G$--orbits share a common accumulation point $x$,
then the closure $\Lambda$ of $Gx$ is the unique minimal invariant set; furthermore, either $\Lambda=\bR$ or 
$\Lambda$ is exceptional.
\item
If $G$ admits an exceptional minimal invariant set $\Lambda$,
then there exists a monotone continuous surjective map $h\co \bR\to\bR$ and a homomorphism $\Phi\co G\to\Homeo^+(\bR)$ such that $\Phi(G)$ is minimal and such that $hg=\Phi(g)h$ for each $g\in G$.
\ee
\end{lem}

\begin{rem}
The group $\Phi(G)\le\Homeo^+(\bR)$ above is called a \emph{minimalization} of $G$.
Note that $h$ maps the closure of each component of $\bR\setminus\Lambda$ to a single point.
Such a map $h$ is called the \emph{devil's staircase} map.
\end{rem}

\begin{lem}\label{lem:prechain}
For a prechain group $G$ as in Setting~\ref{set}, all of the following hold.
\be
\item
There exists $x\in\supp G$ such that every orbit accumulates at $x$.
\item
For each $g\in G$ and for each compact set $A\sse \supp G$, there exists $u\in G'$ such that the actions of $g$ and $u$ agree as functions on $A$.
\item
Every $G$--orbit is a $G'$--orbit.
\item
$G'$ is locally CO-transitive.
\ee
\end{lem}

\bp
We may assume $\supp G=\bR$ as usual.

(1) Every orbit accumulates at $x=\partial^+ f_1$.

(2) Let us write $g=s_k \cdots s_2 s_1$
for some $s_i\in \FF\cup\FF^{-1}$.
There exist open intervals 
$J\sse \supp f_1$ and $K\sse \supp f_n$
such that
\[\partial^-J=\partial^-\supp f_1,\quad \partial^+K=\partial^+\supp f_n\]
and such that for each $i\in\{1,2,\ldots,k\}$ we have
$s_i \cdots s_2 s_1(A) \cap (J\cup K)=\varnothing$.
There exists an $h_i\in G$ such that
$h_i\supp s_i\sse J\cup K$. 
Then the following element has the desired properties:
\[u=\left(\prod_{i=1}^k h_i s_i^{-1} h_i^{-1}\right) g\in G'.\]

(3) This is immediate by applying the part (2) for the case that $A$ is a singleton.

(4) By the part (3), it suffices to show that $G$ is locally CO-transitive.
We let $x=\partial^+\supp f_1$.
Let us fix a compact set $A\sse \bR$
and an open neighborhood $J$ of $x$.
For $h=f_n\cdots f_1$,
we can find $\ell,m\gg0$ such that
$ h^{-m}(A)\sse \supp f_1$ and such that
$
f_1^\ell h^{-m}(A)\sse J.$
\ep

We have the following dichotomy of chain group actions.

\begin{thm}\label{thm:dichotomy}
For a chain group $G$, exactly one of the following holds:
\begin{enumerate}[(i)]
\item
$G$ is minimal; in this case, $G'$ is simple.
\item
$G$ admits a unique exceptional minimal invariant set; in this case, $G$ surjects onto a minimal chain group.
\ee
\end{thm}

\bp
We may assume $\supp G=\bR$.
From Lemmas~\ref{lem:dichotomy} and~\ref{lem:prechain} (1),
we have that either $G$ is minimal or exceptional.

Suppose $G$ is minimal, and write $H=G'$.
Lemma~\ref{lem:prechain} implies that $H$ is minimal and CO-transitive. 
Elements of $H$ have the property that their supports are contained in compact intervals of $\bR$, since the germs of $G$ at $-\infty$ and $\infty$ are abelian quotients of $G$. Hence, $H$ consists of compactly supported homeomorphisms of $\bR$. 
Since $H$ is perfect and nonabelian, Lemma \ref{lem:higman}  implies that $H=H'$ is simple.

Suppose $G$ is exceptional. 
We use the notation from Setting~\ref{set} and Lemma~\ref{lem:dichotomy}.
The map $h$ is two-to-one or one-to-one at each point in $\Lambda$. Since $\Lambda$ accumulates at both boundary points of each interval $\supp f_i$, the supports of the $\Phi(f_i)$ are the non-degenerate open intervals $h(\supp f_i)$, which together still form a chain of intervals for $i=1,2,\ldots,n$.
Since $\form{f_i,f_{i+1}}$ has only abelian proper quotients, we see $\Phi(\form{f_i,f_{i+1}})\cong F$. Hence $\Phi(G)$ is a chain group which acts minimally on its support, namely $\Phi(\Lambda)$.
\ep

It follows that every proper quotient of a minimal chain group is abelian.

\section{Chain groups and subgroups of $\Homeo^+(I)$}\label{sec:fg}

In this section, we prove Theorem \ref{thm:embed}, which is fundamental for establishing many of the remaining results claimed in the introduction.

\subsection{Groups of homeomorphisms}\label{subsec:continuous}
For an interval $J\sse \bR$, recall that $\Homeo^+(J)$ may be regarded as a subgroup of $\Homeo^+(\bR)$, using the extension by the identity outside $J$. We continue to use the notations $a,b\in\Homeo^+(\bR)$ from Section~\ref{s:prelim}. 
Let us note the following independent observation:

\begin{lem}\label{lem:comm-lemma}
If $G\le\Homeo^+(0,1)$ is an $n$--generated group such that $G/G'\cong \bZ^n$ for some $n\ge1$,
then $G$ embeds into $\form{G,a}'\le\Homeo^+(\bR)$.
\end{lem}

\bp
Let us write
\begin{align*}
G&=\form{g_1,\ldots,g_n},\quad G/G'\cong \oplus_{i=1}^n \bZ g_i,
\\
\bar g_i &= a^i g_i a^{-i}\in\Homeo^+(i,i+1),\quad h_i=g_i \bar g_i^{-1}\in\Homeo^+(\bR),\\
H&=\form{h_1,\ldots,h_n}\le \form{G,a}'.
\end{align*}
It is straightforward to see that
$\phi(h_i)=g_i$ defines an isomorphism $H\cong G$.\ep

We define $\AA$ as the set of  $g\in\Homeo^+(\bR)$  such that 
\[ g(x)\ \begin{cases}
= x&\text{ if }x\leq 0,\\
\in (x,x+1) &\text{ if }0<x<1,\\
  =x+1&\text{ if }x\ge 1.
\end{cases}\]
Note that $b\in \AA$.

\begin{lem}\label{lem:chain-lemma}
For each $g_1,\ldots,g_n\in\AA $,
the group $\form{g_1,\ldots,g_n,a}$ is abstractly isomorphic to an $(n+1)$--chain group.
\end{lem}

\bp
Write $B=\{g_1,\ldots,g_n\}$
and define
\begin{align*}
f_0&=g_1^{-1}a,\quad f_n=a^{n-1}g_na^{1-n},\\
f_i&= (a^i g_{i+1}^{-1}a^{-i})(a^{i-1}g_ia^{1-i})\text{ for }1\le i\le n-1,\\
\FF&=\{f_0,\ldots,f_n\},\text{ and }G=\form{\FF}=\form{B,a}.
\end{align*}
We make the following easy observations.

\begin{claim*}
We have the following:
\setcounter{claim}{0}
\be
\item
$\supp f_0=(-\infty,1)$, $\supp f_n=(n-1,\infty)$.
\item
$\supp f_i=(i-1,i+1)$ for $1\le i\le n-1$.
\item
$f_{i+1}f_i(i)=i+1$ for $1\le i\le n-1$.
\ee
\end{claim*}
Now Lemma~\ref{lem:dynamical condition} and the claim above imply that $\FF$ generates $G$ as an $(n+1)$--chain group.
\ep

\begin{cor}\label{cor:2-to-3} Thompson's group $F$ is isomorphic to an $n$--chain group for each $n\ge2$. \end{cor}
\bp Apply Lemma~\ref{lem:chain-lemma} for $g_1=\cdots=g_{n-1}=b\in\AA $.  \ep

\begin{lem}\label{lem:bA}
If $g\in \Homeo^+(1/4,1/2)\le\Homeo^+(\bR)$,
and if $c\in\AA $ satisfies $c(1/4,1/2)=(1/2,1)$, then  $c g\in\AA $.
\end{lem}
\bp
If $x\not\in(1/4,1/2)$, then $cg(x)=c(x)$.
If $x\in (1/4,1/2)$, then
\[
c^{-1}(x) <c^{-1}(1/2)= \frac14 < g(x) <\frac12 =c^{-1}(1) <  c^{-1}(x+1).\qedhere
\]
\ep

\begin{thm}\label{thm:fg-to-chain}
Let $G=\form{h_1,h_2,\ldots,h_n}\le\Homeo^+(\bR)$ for some $n\ge1$.
\be
\item
The group $G$ embeds into an $(n+2)$--chain group $L$ such that $L'$ is simple.
\item
If, moreover, $\supp h_1$ has finitely many components then
$G$ embeds into an $(n+1)$--chain group $L$ such that $L'$ is simple.
\ee
In both of the cases,
if the additional hypothesis $G/G'\cong\bZ^n$ holds, 
then we can further require that $G$ embeds into $L'$.
\end{thm}

\bp
(1)
Write $G=\form{B}$ for some $B=\{h_1,\ldots,h_n\}\sse\Homeo^+(\bR)$.
We may suppose $\supp G\sse (1/4,1/2)$ after a suitable conjugation.
Lemma~\ref{lem:bA} implies that $b B\sse\AA $.
From Lemma~\ref{lem:chain-lemma}, we see that  $L=\form{B,b,a}=\form{bB, b,a}$ is an $(n+2)$--chain group.
Moreover, $L$ acts minimally since $F=\form{a,b}$ does as well. Hence, $L'$ is simple by Theorem~\ref{thm:dichotomy}.

If $G/G'\cong\bZ^n$, then 
Lemma~\ref{lem:comm-lemma} gives us a desired embedding $G\hookrightarrow \form{G,a}'\hookrightarrow L'$.

(2)
We assume $B=\{h_1,\ldots,h_n\}\sse\Homeo^+(1/4,1/2)$, and define
$L=\form{B,a,b}$.
We may further suppose that $h_1\in F=\form{a,b}$, possibly after a conjugation. 
Then \[L=\form{B\setminus\{h_1\},b,a}=\form{b(B\setminus\{h_1\}),b,a}\]
is an $(n+1)$--chain group.
The rest of the proof is identical to the part (1).
\ep

\begin{rem}\label{rem:minimal}
In Theorem~\ref{thm:fg-to-chain}, we build a chain group $G$ containing a two--chain subgroup $F$ which acts minimally. It is straightforward to see that this is indeed enough to guarantee that the whole chain group $G$ acts minimally. We revisit this fact below in Section~\ref{sec:uniqueness}, specifically in Lemma~\ref{lem:loc-dense}.
\end{rem}

Theorem~\ref{thm:embed} is implied by the part (1).

\subsection{Isomorphisms between chain groups}
In this subsection, we prove Proposition \ref{prop:subchain group}, which follows fairly easily from the ideas in Subsection \ref{subsec:continuous}.

\begin{thm}\label{thm:n to m}
For $m\ge n\ge 2$, every $n$--chain group is isomorphic to an $m$--chain group.
\end{thm}

\begin{proof}
The case $n=2$ follows from Corollary~\ref{cor:2-to-3}. Assume $n\ge3$ and let $\{f_1,f_2,\ldots,f_n\}$ generate $G$ as an $n$--chain group. We will show that $G$ is isomorphic to an $(n+1)$--chain group, which will establish the result by an easy induction. We consider the rightmost three intervals in the chain and we write $p=f_{n-2}$, $q=f_{n-1}$, and $r=f_n$.
We assume each generator moves points to the right in its support.

We let $M$ be a sufficiently large integer which will be determined later.
Define 
\[d=q^{-M}pq^M,\quad f=(pr)^M q(pr)^{-M},\quad e = f^{-M}qf^M.\] 
Notice that 
$d,e,f,r$ generate a $4$--prechain group, as shown in Figure~\ref{f:convert}.

Lemma~\ref{lem:dynamical condition} implies that for $M\gg0$,
we have
$\form{e,f^M}\cong\form{q,f^M}\cong\form{f^M,r}\cong F.$
Moreover, we have
\begin{align*}
&\form{f_{n-3},d}\cong{q^{M}}\form{f_{n-3},d}{q^{-M}}=\form{f_{n-3},p}\cong F,\\
&\form{d,e}\cong{f^{M}}\form{d,e}{f^{-M}}=\form{d,q}=\form{p,q}\cong F,\\
&\form{d,e,f^M,r}=\form{d,q,f^M,r}=\form{p,q,f^M,r}=\form{p,q,r}.\end{align*}
Finally, notice that the left endpoints of $\supp e$ and $\supp q$ coincide. Therefore, $G$ is an $(n+1)$--chain group generated by 
\[f_1,\ldots,f_{n-3},d,e,f^M,r.\qedhere\]
\ep

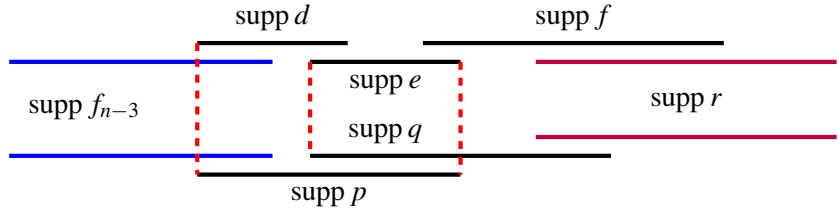
\begin{figure}[htb!]
\begin{tikzpicture}[ultra thick,scale=.5]
\draw [blue] (-10,0) -- (-3,0);
\draw (-8,-.5) node [below] {\small $\supp f_{n-3}$}; 
\draw (-2,0) -- (2,0);
\draw (0,0) node [below] {\small $\supp e$}; 
\draw [purple] (4,0) -- (12,0);
\draw (8,-.5) node [below] {\small $\supp r$};

\draw (-5,.5) -- (-1,.5);
\draw (-3,.5) node [above] {\small $\supp d$}; 

\draw (1,.5) -- (9,.5);
\draw (5,.5) node [above] {\small $\supp f$}; 
\draw (-5,-3) -- (2,-3);
\draw (-1.5,-3) node [below] {\small $\supp p$}; 
\draw [purple] (4,-2) -- (12,-2);

\draw (-2,-2.5) -- (6,-2.5);
\draw (0,.-2.5) node [above] {\small $\supp q$}; 
\draw [blue] (-10,-2.5) -- (-3,-2.5);

\draw [dashed,red] (-5,.5) -- (-5,-3);
\draw [dashed,red] (-2,0) -- (-2,-2.5);
\draw [dashed,red] (2,0) -- (2,-3);
\end{tikzpicture}
\caption{Converting an $n$--chain group to an $(n+1)$--chain group.}
\label{f:convert}
\end{figure}

From the presentation given in Section~\ref{ss:thompson}, it is easy to check that the abelianization of the $n^{th}$ Higman--Thompson group $F_n$ is exactly $\bZ^n$. From Theorem~\ref{thm:n to m}, we quickly obtain a proof of Proposition~\ref{prop:abelianization}:

\begin{proof}[Proof of Proposition~\ref{prop:abelianization}]
Let $n\geq 2$ and let $2\leq k\leq n$ be an integer. By Theorem~\ref{thm:n to m}, there is an $n$--chain group $G_k$ such that $G_k$ is isomorphic to the $k^{th}$ Higman--Thompson group $F_k$, so that $H_1(G_k,\bZ)\cong\bZ^k$.
\end{proof}

\subsection{Chain groups with non-simple commutator subgroups}

\begin{prop}\label{prop:blow-up}
For each $n\ge3$, there exists an $n$--chain group with a non-simple commutator subgroup.
\end{prop}

\bp
Let us describe an example of a $3$--chain group with a non-simple commutator subgroup.
We begin by finding a $3$--chain group whose minimalization map is not injective. 
We will find such an example by blowing-up (sometimes called \emph{Denjoy-ing}  after \cite{Denjoy1932}) an orbit.
The case $n\ge4$ will then follow by Theorem~\ref{thm:n to m}.

By Lemma~\ref{lem:chain-lemma},  the following generate $F\le\Homeo^+(\bR)$ as a $3$--chain group:
\[f_0=b^{-1}a,\quad f_1=ab^{-1}a^{-1}b,\quad f_2=aba^{-1}.\]
Pick the orbit $\OO=F(0)$, and let $\{\ell_y\co y\in\OO\}\sse\bR_{>0}$ satisfy
$\sum_{y\in\OO}\ell_y=1$.
Denote the closed interval $[0,\ell_y]$ by $I_y$.

We now replace each point $y\in\OO$ by the interval $I_y$.
Formally, we let
\[
\bR^\sim=(\bR\setminus\OO)\coprod\left(\coprod_{y\in\OO}I_y\right)\]
Define $q\co \bR^\sim\to\bR$ by  $q(x)=x$ for $x\in\bR\setminus\OO$ and $q(I_y)=\{y\}$,
and topologize $\bR^\sim$ by the devil's staircase map $q$. Note $\bR\approx \bR^\sim$.
We have a natural map
\[
\Phi\co \Homeo^+_{\bR\setminus\OO}(\bR^\sim)\to 
\Homeo^+_\OO(\bR)\]
from the set of $\bR\setminus\OO$--preserving homeomorphisms of $\bR^\sim$
to the set of $\OO$--preserving  homeomorphisms of $\bR$. The map $\Phi$ is 
 determined by the condition
\[q\circ f=\Phi(f)\circ q\quad\text{ for }f\in
\Homeo^+_{\bR\setminus\OO}(\bR^\sim).\]

For each $y,z\in\OO$, we will fix a homeomorphism $h^y_z\co I_y\to I_z$ which is \emph{equivariant} in the following sense (for example, we can choose $h^y_z$ to be linear):
\[
h^y_y=\Id\quad\text{and}\quad h^y_z \circ h^x_y=h^x_z\quad\text{for}\quad x,y,z\in\OO.\]
Then for $f\in \Homeo^+_\OO(\bR)$, the equation
\[\phi(f)=\prod_{y\in\OO} h^y_{f(y)}\]
defines a section $\phi$ of $\Phi$.

Write 
$g_i=\phi(f_i)$
for $i=0,2$.
We now pick an arbitrary homeomorphism $h$ such that $\supp h=I_1$,
and let $g_1=h\circ \phi(f_1)$. 
Note that $[h,g_0]=[h,g_2]=1$.

\setcounter{claim}{0}
\begin{claim}
The group $G=\form{g_0,g_1,g_2}$ is a three--chain group.
\end{claim}
We note $G$ is a prechain group as illustrated in Figure~\ref{f:blow-up}.
The claim follows from the dynamical condition of Lemma~\ref{lem:dynamical condition}.

\begin{figure}[htb!]
\begin{tikzpicture}[ultra thick,scale=.5]
\draw (-9,1.5)--(-7,1.5) (-5,1.5)--(-1,1.5);
\draw [blue] (-1,2.5)--(1,2.5) (-7,2.5)--(-5,2.5)   (7,2.5)--(5,2.5); 
\draw (-8,2.5) node [blue] {\small $\cdots$};
\draw (8,2.5) node [blue] {\small $\cdots$};
\draw (-3,2.5) node [blue] {\small $\cdots$};
\draw (3,2.5) node [blue] {\small $\cdots$};
\draw (9,1.5)--(7,1.5)  (5,1.5)--(1,1.5);
\draw (-6,2.5) node [above] {\small $I_0$};
\draw (0,2.5) node [above] {\small $I_{1}$};
\draw (6,2.5) node [above] {\small $I_2$};
\draw (10,3) node [right] {\small $\coprod_y I_y$};
\draw (10,1.5) node [right] {\small $\bR\setminus\OO$};
\draw (9,0) node [right] {\small $\supp g_2$};
\draw (-9,0) node [left] {\small $\supp g_0$};
\draw (0,-.5) node [below] {\small $\supp g_1$};
\draw (0,.3) node [above] {\small $\supp h$};
\draw [red] (-9,0) -- (-1,0);
\draw [teal] (-5,-.5)--(5,-.5);
\draw [purple] (9,0) -- (1,0);
\draw [cyan] (-1,.3) -- (1,.3);
\end{tikzpicture}
\caption{Blowing-up an orbit $\OO$. Note that $\bR\setminus\OO$ is totally disconnected.}
\label{f:blow-up}
\end{figure}
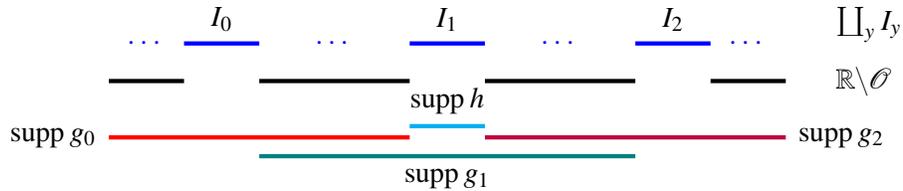

\begin{claim}\label{cla:prop} 
We have $1\ne h\in\ker\Phi\cap G$.
\end{claim}
The part $1\ne h\in\ker\Phi$ follows from the definition.
To show $h\in G$, we note
\[g_1g_0g_2g_1(g_2g_1g_0)^{-1}
=h^2\phi((f_1f_0)(f_2f_1)(f_2f_1f_0)^{-1})h^{-1}
=h^2\phi(ab^{-1}\cdot ba^{-1})h^{-1}=h.\]
So, $\Phi$ is not injective on $G$.

\begin{claim} 
We have $1\ne \ker\Phi\cap G'$.
\end{claim}

Since $1\neq h\in\ker\Phi$, we have that all conjugates of $h$ by elements of $G$ also lie in $\ker\Phi$ and are nontrivial. Choose an element $g\in G$ such that $g(I_1)=I_0$. Then $\supp gh^{-1}g^{-1}=I_0$, and the element $hgh^{-1}g^{-1}$ is nontrivial and lies in $G'$.
\ep

It is straightforward to blow-up an orbit of a minimal chain group to obtain a non-minimal chain group which is algebraically isomorphic to the original one. The proof of Proposition~\ref{prop:blow-up} shows how to alter the isomorphism type by a blow-up.

\section{Uncountability results}\label{sec:uncountable}

\subsection{Uncountable families of isomorphism types}\label{sec:neumann}
We first establish the existence of uncountably many isomorphism types of subgroups of $\Homeo^+(\R)$.

\begin{lem}\label{lem:left orderable}
There exists a two--generated left orderable group $\Gamma$ and a collection of normal subgroups $\{N_i\}_{i\in I}$ of $\Gamma$ with the following properties:
\begin{enumerate}
\item
The collection $\{N_i\}_{i\in I}$ is uncountable;
\item
For each $i$, the group $N_i<\Gamma$ is central;
\item
For each $i\in I$, the quotient $\Gamma_i=\Gamma/N_i$ is left orderable.
\end{enumerate}
\end{lem}

The group $\Gamma$ and its subgroups as in Lemma \ref{lem:left orderable} appear in III.C.40 of de la Harpe's book~\cite{MR1786869}. Here, we merely observe that $\Gamma$ and the quotients $\Gamma_i$ are all left orderable.

\begin{proof}[Proof of Lemma \ref{lem:left orderable}]
To establish the first two claims, we reproduce the argument given by de la Harpe nearly verbatim. Let $S=\{s_i\}_{i\in\bZ}$, and let \[R=\{[[s_i,s_j],s_k]=1\}_{i,j,k\in\bZ}\cup\{[s_i,s_j]=[s_{i+k},s_{j+k}]\}_{i,j,k\in\bZ}.\] Define $\Gamma_0=\langle S\mid R\rangle$ and let $\Gamma=\Gamma_0\rtimes\bZ$, where the conjugation action of $\bZ=\langle t\rangle$ is given by $t^{-1}s_it=s_{i+1}$. For each $i$, we set $u_i=[s_0,s_i]$. Note the following easy observations:
\begin{enumerate}[(i)]
\item
The group $[\Gamma_0,\Gamma_0]$ is central in $\Gamma$, is generated by $\{u_i\}_{i\in\bZ}$, and is isomorphic to an infinite direct sum of copies of $\bZ$; 
\item
The group $\Gamma$ is generated by $s_0$ and $t$;
\item
The quotient group $\Gamma/[\Gamma_0,\Gamma_0]$ is isomorphic to the lamplighter group $\bZ\wr\bZ$.
\end{enumerate}

For each subset $X\subset \bZ\setminus\{0\}$, we can consider the group $N_X=\langle u_i\mid i\in X\rangle$. Evidently these groups are distinct for distinct subsets of $\bZ\setminus\{0\}$, and they are all central (and hence normal) because $[\Gamma_0,\Gamma_0]$ is central in $\Gamma$. We thus establish the first two claims of the lemma.

For the third claim, note that the lamplighter group $\Z\wr\Z$ is left orderable since it lies as a subgroup of $F$ (see Lemma \ref{lem:lamplighter}). For each $X\subset\bZ\setminus\{0\}$, the groups $[\Gamma_0,\Gamma_0]/N_X$ are all free abelian and therefore left orderable. By Lemma \ref{lem:extension}, it follows that the group $\Gamma/N_X$ is left orderable.
\end{proof}

\begin{lem}\label{lem:uncountable}
There exist uncountably many isomorphism types of two--generated subgroups of $\Homeo^+(\R)$.
\end{lem}
\begin{proof}
By Lemma \ref{lem:dynamical realization}, it suffices to prove that there are uncountably many isomorphism types of two--generated left orderable groups. To this end, suppose there exist only countably many isomorphism types of two--generated left orderable groups. Then the class of groups $\mathcal{N}=\{\Gamma/N_X\}_{X\sse\bZ}$ furnished by Lemma \ref{lem:left orderable} consists of only countably many isomorphism types. It follows that there is an element $N\in\mathcal{N}$ and uncountably many surjective homomorphisms $\Gamma\to N$. Since $N$ and $\Gamma$ are both finitely generated, this is a contradiction.
\end{proof}

Note that the case $n\ge 4$ of Theorem \ref{thm:uncountable} follows from Lemma~\ref{lem:uncountable}.

\subsection{Uncountably many isomorphism types of $3$--chain groups}\label{sec:three chain}
We now prove Theorem \ref{thm:uncountable} in the case $n=3$. We retain notation from Subsection \ref{sec:neumann} and write $\mathcal{N}=\{\Gamma/N_X\}_{X\subset\Z}$. As before, each $N\in\mathcal{N}$ is generated by two elements, $s(=s_0)$ and $t$.

\begin{lem}\label{lem:fp free}
Let $N\in\mathcal{N}$ be generated by elements $s,t\in N$. There exists a faithful action of $N$ on $\bR$ such that the element $t\in N$ acts without fixed points.
\end{lem}

We thank the referee for sketching the following conceptual proof of the lemma, which is much easier to understand than the original provided by the authors. Before giving the proof, we recall two notions from orderability of groups: let $\prec$ be a left--invariant order on a group $G$. We say that an element $t\in G$ is \emph{cofinal} if for all $g\in G$ there is an $n\in\bZ$ such that $t^{-n}\prec g\prec t^n$. A subgroup $H<G$ is called \emph{convex} if whenever $a\prec b\prec c$ for $a,c\in H$ and $b\in G$, then $b\in H$.

There is a dictionary between order--theoretic notions and dynamical notions, which was developed by Navas in~\cite{NavasAIF10}. According to that dictionary, it suffices to find a left ordering on $N\in\mathcal{N}$ such that the element $t$ is cofinal with respect to that ordering.

\begin{proof}[Proof of Lemma~\ref{lem:fp free}]
Consider the natural surjection $N\to\Z\wr\Z$. Recall that then $\Z\wr\Z$ is a semidirect product of a copy $\langle t\rangle\cong\Z$ with a copy of $\oplus_{i=1}^{\infty}\Z$, where one of these copies of $\Z$ is generated by an element $s$. We have the following short exact sequence decomposition \[1\to\oplus_{i=1}^{\infty}\Z\to\Z\wr\Z\to\Z\to 1.\] Arbitrary orderings on the rightmost copy of $\Z$ and on $\oplus_{i=1}^{\infty}\Z$ combine to give a unique ordering on $\Z\wr\Z$ as in Lemma~\ref{lem:extension}. It is straightforward to check that with respect to such an ordering on $\Z\wr\Z$, the subgroup $\oplus_{i=1}^{\infty}\Z$ will be convex, and the generator $t$ of the rightmost $\Z$ will cofinal.

Now, the kernel of the surjection $N\to \Z\wr\Z$ is the group $K=[\Gamma_0,\Gamma_0]/N_X$, which being a direct sum of copies of $\Z$, if left orderable. Therefore, we may choose an arbitrary ordering on $K$ and build an ordering on $N$ from the short exact sequence \[1\to K\to N\to\Z\wr\Z\to 1.\] Again, it is straightforward to check that $K$ will become a convex subgroup in this ordering on $N$, and any cofinal element of the ordering on $\Z\wr\Z$ will remain cofinal in the ordering on $N$.
\end{proof}

Theorem~\ref{thm:fg-to-chain} and Lemma~\ref{lem:fp free} trivially imply the following:

\begin{lem}\label{lem:n-simple}
Each $N\in\NN$ embeds into the commutator subgroup of a minimal 3-chain group.
\end{lem}

For a more detailed discussion concerning why the resulting chain group is minimal, we refer the reader to Remark~\ref{rem:minimal} above and to Section~\ref{sec:uniqueness} below.

\bp[Proof of Theorem \ref{thm:uncountable}, in the case $n\geq 3$]
Note that the commutator groups arising from the conclusion of Lemma~\ref{lem:n-simple} are simple by Theorem~\ref{thm:dichotomy}.
Let $\CC$ be the class of the isomorphism types of $n$--chain groups that have simple commutator subgroups, and $\CC_0\sse\CC$ be the subclass consisting of finitely presentable ones. Then $\CC$ is uncountable from Lemma~\ref{lem:n-simple} and from that each countable group admits countably many isomorphism types of finitely generated subgroups.
Since $\CC_0$ is obviously countable, we have that $\CC\setminus\CC_0$ is uncountable.
\end{proof}

\bp[Proof of  Corollary \ref{cor:simple}]
The commutator subgroup of each group in $\CC$ has at most countably many finitely generated subgroups. 
Since each $N\in\NN$ embeds into some group $\CC$, we have the desired conclusion.
\ep

\subsection{Smoothability of chain groups}
We conclude this section with a remark on the diversity of non-$C^2$--smoothable and non-$C^1$--smoothable chain groups.
Recall the definition of the integral Heisenberg group: $H=\langle x,y,z\mid [x,y]z^{-1}=[x,z]=[y,z]=1\rangle$.

\begin{prop}\label{prop:heisenberg}
There exist uncountably many distinct isomorphism types of $3$--chain groups which admit no faithful $C^2$ actions on $[0,1]$.
\end{prop}

\bp
Let $X\subset\Z\setminus\{0\}$ be proper, and let $G=\Gamma/N_X\in\mathcal{N}$.
Fix $i\in \bZ\setminus(X\cup\{0\})$. Then we have a map $j\co H\to G$ defined by 
$(x,y,z)\mapsto (s_0,s_i,[s_0,s_i])$.
Since the abelianization of $\Gamma_0$ is torsion--free and freely generated by the generators $\{s_i\mid i\in\Z\}$, we have that $\ker j\le H'=Z(H)=\langle z\rangle$. If $j$ is not injective then the image of $z$ in $\Gamma_0$ is either trivial or has finite order. However, $u_i$ has infinite order in $\Gamma_0/N_X$, so that $j$ must be injective.

The Plante--Thurston Theorem~\cite{PT1976} states that every nilpotent subgroup of $\Diff^2(I)$ is abelian.
Since $H$ is nonabelian nilpotent
and embeds into $G$, we see $G$ does not admit a faithful $C^2$--action on $[0,1]$. The proposition follows from Lemma~\ref{lem:n-simple}.
\end{proof}

We now give a proof of Corollary~\ref{cor:c1}. Recall that in~\cite{LodhaMoore}, the third author and Moore studied a certain group of PL homeomorphisms of the interval, now known as the Lodha--Moore group. We shall require two facts about this group: the first is that it is $3$--generated and its abelianization is isomorphic to $\bZ^3$, as follows easily from the presentation given in~\cite{LodhaMoore}. The second fact is the following result of the third author with Bonatti and Triestino, which appears as~\cite[Theorem 3.2]{BLT2017}:

\begin{thm}\label{thm:BLT}
The Lodha--Moore group $G$ admits no nonabelian $C^1$ action on the interval.
\end{thm}

\begin{proof}[Proof of Corollary~\ref{cor:c1}]
Let $N\in\mathcal{N}$ and let $G$ be the Lodha--Moore group. Then $N\times G$ is a $5$--generated subgroup of $\Homeo^+(I)$ which admits no faithful $C^1$ action on $I$, by Theorem~\ref{thm:BLT}, and whose abelianization is isomorphic to $\bZ^5$. By Theorem~\ref{thm:fg-to-chain} and Lemma~\ref{lem:fp free}, we have that $N\times G$ embeds into the commutator subgroup of a $6$--chain group, and moreover we may assume that this commutator subgroup is simple. Since every group of the form $N\times G$ is $5$--generated and since these groups fall into uncountably many isomorphism types as $N$ varies in $\mathcal{N}$, we obtain the desired conclusion of the corollary.
\end{proof}

\section{Local density of actions and the uniqueness of chain group actions}\label{sec:uniqueness}

In this section, we establish Theorem~\ref{thm:uniqueness}.

\subsection{Minimality, local density, and Rubin's Theorem}

Let $G$ be a group acting on a topological space $X$. We say that the action of $G$ on $X$ is \emph{minimal} if every $G$--orbit is dense in $X$. If $U\subset X$ is open, we write $G_U=\{g\in G\mid \supp g\subset U\}$, a subgroup of $G$ called the \emph{rigid stabilizer} of $U$. A subset $A\subset X$ is called \emph{non--nowhere dense} if the closure of $A$ in $X$ has nonempty interior.

The action of $G$ on $X$ is \emph{locally dense} if for each $x\in X$ and any neighborhood $U$ of $x$, the orbit $G_U.x$ is non--nowhere dense.

\begin{lem}\label{lem:minimal-dense}
Let $X$ be a connected topological space and let $G<\Homeo(X)$. Then $G$ is minimal if and only if each orbit of $G$ is non--nowhere dense.
\end{lem}
\begin{proof}
If $G$ is minimal then the closure of each orbit is equal to $X$ and is hence non--nowhere dense.

Suppose conversely that each orbit of $G$ is non--nowhere dense. 
For each $x$,  let 
\[Y_x = \overline{G.x},\quad U_x = \mathrm{int}(Y_x).\]
Notice that $Y_x$ and $U_x$ are nonempty and $G$--invariant.
Since $G.x$ is dense in $Y_x$, we have $U_x\cap G.x\ne\varnothing$.
So, $U_x$ is an open neighborhood of $x$.

Now let $x\in X$ and $t\in Y_x$ be arbitrary.
Then we have
\[
t\in U_t\sse Y_t=\overline{G.t}\sse Y_x.\]
This implies that $t\in\mathrm{int} (Y_x)$.
Since $Y_x$ is closed and open, we see $X=Y_x$. The lemma follows.
\end{proof}

Rubin's Theorem provides the connection between locally dense actions of groups and topological conjugacy:

\begin{thm}[Rubin's Theorem, cf.~\cite{BrinGD04,Rubin96}]\label{thm:rubin}
Let $G<\Homeo(X)$ and $H<\Homeo(Y)$ be groups of homeomorphisms of locally compact Hausdorff topological spaces $X$ and $Y$ respectively, and assume that neither $X$ nor $Y$ has isolated points. Suppose furthermore that $G$ is isomorphic to $H$ and that the actions of $G$ and $H$ are locally dense. Then an isomorphism $\phi\colon G\to H$ induces a unique homeomorphism $h\colon X\to Y$ such that $\phi(g)=h\circ g \circ h^{-1}$.
\end{thm}

\subsection{Minimal chain group actions}

We now specialize to the case of a chain group $G<\Homeo^+(I)$, with the assumption that $\supp G=(0,1)\subset [0,1]=I$. Throughout this section, we say that $G$ acts minimally on $I$ if it acts minimally on $(0,1)$.

\begin{lem}\label{lem:loc-dense}
For a chain group $G$ acting on $I$, the following conditions are equivalent:
\begin{enumerate}
\item
The action of $G$ is minimal;
\item
The diagonal action of $G$ on the set \[X=\{(x,y)\in (0,1)\times (0,1)\mid x<y\}\] is minimal;
\item
Each orbit of $G$ is non--nowhere dense;
\item
The action of $G$ is locally dense;
\item
For each connected open set $U\subset I$, the group $G_U$ acts minimally on $U$;
\item
For each open set $U\subset I$, we have $\supp G_U=U$.
\end{enumerate}
\end{lem}

\begin{proof}
Lemma~\ref{lem:minimal-dense} implies the equivalence of (1) and (3), as well as of (4) and (5). We have that (2) implies (1) trivially.

{\bf (1) implies (2).} Assume that $G$ acts minimally on $(0,1)$ and let $(x_1,y_1)$ and $(x_2,y_2)$ be given. First, we may apply an element of $G$ to $(x_1,y_1)$ to get a pair $(x,y)$ such that $|x-x_2|$ is as small as we wish. There are two cases to consider now: $x<y_2<y$ and $x<y<y_2$.

We treat the case $x<y_2<y$ first. Let $f_k$ be the generator of the chain group $G$ such that $\sup\supp f_k=1$. We then have $\supp f_k=(a,1)$. By the minimality of the action of $G$ on $I$, we may replace $f_k$ by a suitable conjugate $g_k$ such that $\supp g_k=(b,1)$ for some $x<b<y_2<y$, with $|b-y_2|$ as small as we wish. Then, replacing $g_k$ by its inverse if necessary, we have that $g_k$ fixes $x$ and, $g_k^n(y)\to b$ as $n\to\infty$. It follows in this case that $(x_2,y_2)$ is in the $G$--orbit closure of $(x_1,y_1)$.

To treat the case $x<y<y_2$, we retain the setup of the previous case, and note that $g_k^{-n}(y)\to 1$ as $n\to\infty$. Thus, we may apply a sufficiently negative power of $g_k$ to $y$ so that \[g_k^{-n}(x)=x<y<y_2<g_k^{-n}(y).\] We have thus reduced the situation to the first case we have already treated, whence (1) implies (2).

{\bf (2) implies (4).} Suppose (2) holds, and let $U\subset I$ be an open set. Since $G$ is a chain group, it is straightforward to find an element $g\in G$ whose support is compactly contained in $(0,1)$, i.e. $\supp g\subset K\subset (0,1)$ for some compact subset $K$. If $x\in U$, then (2) implies that some suitable conjugate $g_U$ of $g$ satisfies $x\in\supp g_U\subset U$. In particular, $G_U$ is nontrivial.

We claim that the $G_U$--orbit of $x$ is dense in $U_0$, the connected component of $U$ containing $x$. If $y\in U_0$, then (2) again implies that for all $\epsilon>0$ there is a conjugate $g_{y,\epsilon}$ of $g_U$ such that $x\in\supp g_{y,\epsilon}\subset U$ and such that $d(y,\partial\supp g_{y,\epsilon})<\epsilon$. Since under the action of $g_{y,\epsilon}$ the orbit of $x$ accumulates on both points in $\partial\supp g_{y,\epsilon}$, we have that $y$ lies in the $G_U$--orbit closure of $x$. It follows that the action of $G$ is locally dense, so that (2) implies (4).

{\bf (5) implies (6).} Suppose that $G_{U_0}$ acts minimally on $U_0$ for each connected subset $U_0\subset (0,1)$, and let $U\subset (0,1)$ be open. Then for each $x\in U$ there is an element $g_x\in G_U$ such that $g_x(x)\neq x$. It follows that $\supp G_U=U$, so that (5) implies (6).

{\bf (6) implies (1).} 
Let $0<x<y<1$.
From $\supp G_{(0,y)}=(0,y)$  we have
\[ y=\sup \{g.x\mid g\in G_{(0,y)}\}.\]
It follows that $y$ is in the closure of $G.x$. This implies the minimality of $G$.
\end{proof}

We now give a proof of Theorem~\ref{thm:uniqueness}, as claimed in the introduction.
\begin{proof}[Proof of Theorem~\ref{thm:uniqueness}]
Let $G,H<\Homeo^+(I)$ be chain groups acting minimally on $(0,1)$. Then Lemma~\ref{lem:loc-dense} implies that the action of $G$ is locally dense. Rubin's Theorem (Theorem~\ref{thm:rubin}) implies that every isomorphism between $G$ and $H$ is induced by a homeomorphism of $I$ intertwining the actions of $G$ and $H$, i.e. a topological conjugacy.
\end{proof}

Proposition~\ref{prop:blow-up} furnishes chain groups whose actions on $I$ have wandering intervals, whereby Theorem~\ref{thm:uniqueness} does not apply. Collapsing the wandering intervals for such a chain group $G$ (i.e. minimalization of the action) furnishes a semi--conjugacy between the action of $G$ and another chain group $H$, whose natural action on $I$ is minimal. The content of Proposition~\ref{prop:blow-up} is that this semi--conjugacy may not induce an isomorphism of groups. Since $2$--chain groups are always isomorphic to $F$, such semi--conjugacies always induce isomorphisms of groups for $2$--chain groups.

\section*{Acknowledgements}
The authors thank C. Bleak, M. Brin, V. Guirardel, J. McCammond, J. Moore, and S. Witzel for helpful discussions. The authors are particularly grateful to A. Navas for several very insightful comments which greatly improved the paper. The authors thank an anonymous referee for many constructive comments. The second and third author thank the hospitality of the Tata Institute of Fundamental Research in Mumbai, where this research was initiated. The authors thank the hospitality of the Mathematical Sciences Research Institute in Berkeley, where this research was completed. 
The first author is supported by Samsung Science and Technology Foundation (SSTF-BA1301-06).
The second author is partially supported by Simons Foundation Collaboration Grant number 429836, by an Alfred P. Sloan Foundation Research Fellowship, and by NSF Grant DMS-1711488.
The third author is funded by an EPFL-Marie Curie fellowship.


\def\cprime{$'$} \def\soft#1{\leavevmode\setbox0=\hbox{h}\dimen7=\ht0\advance
  \dimen7 by-1ex\relax\if t#1\relax\rlap{\raise.6\dimen7
  \hbox{\kern.3ex\char'47}}#1\relax\else\if T#1\relax
  \rlap{\raise.5\dimen7\hbox{\kern1.3ex\char'47}}#1\relax \else\if
  d#1\relax\rlap{\raise.5\dimen7\hbox{\kern.9ex \char'47}}#1\relax\else\if
  D#1\relax\rlap{\raise.5\dimen7 \hbox{\kern1.4ex\char'47}}#1\relax\else\if
  l#1\relax \rlap{\raise.5\dimen7\hbox{\kern.4ex\char'47}}#1\relax \else\if
  L#1\relax\rlap{\raise.5\dimen7\hbox{\kern.7ex
  \char'47}}#1\relax\else\message{accent \string\soft \space #1 not
  defined!}#1\relax\fi\fi\fi\fi\fi\fi}
\providecommand{\bysame}{\leavevmode\hbox to3em{\hrulefill}\thinspace}
\providecommand{\MR}{\relax\ifhmode\unskip\space\fi MR }
\providecommand{\MRhref}[2]{%
  \href{http://www.ams.org/mathscinet-getitem?mr=#1}{#2}
}
\providecommand{\href}[2]{#2}

\end{document}